\documentclass[11pt,a4paper]{article}
\title{%
	\textsc{%
		Equality case in van der Corput's inequality
		and collisions in multiple lattice tilings
	}
}

\author{%
	Gennadiy Averkov\footnote{Faculty of Mathematics, Otto-von-Guericke-Universit\"at Magdeburg, Universit\"atsplatz 2, 39106 Magdeburg, Germany. Email: averkov@ovgu.de}%
}

\usepackage{amsthm,amsmath	,amssymb,enumerate,graphicx,graphpap,xcolor}
\usepackage{mathtools}
\usepackage{empheq}
\usepackage{paralist}
\usepackage{color}
\usepackage{soul}
\usepackage{tikz}
\usepackage{bbm} 

\usepackage[format=plain,labelfont={bf,it},textfont=it]{caption}

\newcommand{\mult}{\operatorname{mult}}

\newcommand{\dd}{\operatorname{d}}

\newcommand{\rmcmd}[1]{\mathop{\mathrm{#1}}\nolimits}

\newcommand{\dotvar}{\,\cdot\,}

\newcommand{\R}{\mathbb{R}}
\newcommand{\N}{\mathbb{N}}
\newcommand{\Z}{\mathbb{Z}}

\newcommand{\floor}[1]{\left\lfloor #1 \right\rfloor}

\newcommand{\cK}{\mathcal{K}}

\newcommand{\cT}{\mathcal{T}}

\newcommand{\cKo}{\cK_o}

\newcommand{\chf}{\mathbbm{1}}

\newcommand{\setcond}[2]{\left\{#1 \, : \, #2 \right\}}

\newcommand{\vol}{\rmcmd{vol}}

\newcommand{\intr}[1]{\operatorname{int}(#1)}
\newcommand{\cl}[1]{\operatorname{cl}(#1)}

\colorlet{myGreen}{green!40!white}

\newtheorem{nn}{}
\newtheorem{theorem}[nn]{Theorem}

\newtheorem{proposition}[nn]{Proposition}
\newtheorem{lemma}[nn]{Lemma}
\newtheorem{corollary}[nn]{Corollary}

\newtheorem*{acknowledgments*}{Acknowledgments}

\begin{document}

\maketitle

\begin{abstract}	
	Van der Corput's provides the sharp bound $\vol(C) \le m 2^d$ on the volume of a $d$-dimensional origin-symmetric convex body $C$ that has $2m-1$ points of the integer lattice in its interior. For $m=1$, a characterization of the equality case $\vol(C)= m 2^d$ is equivalent to the well-known problem of characterizing tilings by translations of a convex body. It is rather surprising that so far, for $m \ge 2$, no characterization of the equality case has been available, though a hint to the respective characterization problem can be found in the 1987 monograph of Gruber and Lekkerkerker. We give an explicit characterization of the equality case for all $m \ge 2$. Our result reveals that, the equality case for $m \ge 2$ is more restrictive than for $m=1$. We also present consequences of our characterization in the context of multiple lattice tilings.
\end{abstract}

\section{Introduction}

Let $d \in \N$. We define a \emph{convex body} as a $d$-dimensional compact convex subset of $\R^d$.
By $\cK^d$ and $\cKo^d$ we denote the family of convex bodies in $\R^d$ and its subfamily consisting of convex bodies centrally symmetric with respect to the origin, repsectively. A set of the form $\Lambda:=\setcond{z_1 u_1 + \cdots + z_k u_k}{z_1,\ldots,z_k \in \Z}$, where $u_1,\ldots,u_k \in \R^d$ are linearly independent, is called a \emph{lattice} of rank $k $, while the $k$-dimensional volume of $\setcond{x_1 u_1 + \cdots + x_k u_k}{x_1,\ldots,x_k \in [0,1]}$ is called the \emph{determinant} of $\Lambda$ and is denoted by $\det(\Lambda)$. A large part of geometry of numbers studies properties of $\cK^d$ and $\cKo^d$ related to lattices. For more on the background, we refer to the monographs \cite{MR893813} and \cite{MR1434478}. In the context of this paper, one can fix the underlying lattice to be the the integer lattice $\Z^d$. We will refer to the elements of $\Z^d$ as \emph{lattice points} or \emph{lattice vectors}. Since every $K \in \cKo^d$ has a positive odd number of interior lattice points,  $\cKo^d$ can be decomposed into disjoint union of families $\cKo^d(2m-1)$, with $m \in \N$, where $\cKo^d(2m-1)$ consists of convex bodies  $K  \in \cKo^d$ that have $2m-1$ interior lattice points. 

\emph{Van der Corput's inequality} is the following useful relation between the volume and the number of interior lattice points in $\cKo^d$: 
\begin{align}
	\label{vdc:ineq}
	\vol(C) & \le m 2^d  \qquad \text{for all} \ C \in \cKo^d(2m-1).
\end{align}
The special case $m=1$, known as the \emph{convex body theorem of Minkowski}, was used as a tool in a multitude of contexts ranging from number theory and algebra to integer optimization. Van der Corput's inequlity the general case $m \in \N$ was used in the theory of lattice polytopes: starting from a seminal work of Hensley \cite{MR688412}, van der Corput's inequality was a basic ingredient in deriving upper bounds on the volume of lattice polytopes in terms of their number of interior lattice points; see \cite{MR1138580,MR1996360,MR3318147,AKN2017}. 

The inequality is sharp, as $m 2^d$ is the maximum volume of convex bodies in $\cKo^d(2m-1)$. Consider for example the `stretched box' $C=[-m,m] \times [-1,1]^{d-1}$, for which the maximum volume $m 2^d$ is attained.  Having a sharp inequality, it is natural to wonder about a possible characterization of its equality case. It is quite surprising that the equality case of van der Corput's inequality has not yet been studied. 
 In author's option, such a characterization must have applications in theory of lattice polytopes, and this was author's original source of motivation. We also mention that \eqref{vdc:ineq} has a discrete counterpart, which has been derived in \cite{MR3480976}, and for which the authors of \cite{MR3480976} have characterized the equality case (see also \cite{MR3022127} for a related result).

Below we give a short summary of what has been known about the equality case of \eqref{vdc:ineq}. With each $A \subseteq \R^d$, one can associate the family
\[
	\cT(A):= \setcond{A + z}{z \in \Z^d}
\]
of translations of $A$ by the vectors of the integer lattice. For $K \in \cK^d$, the family $\cT(K)$ is called an \emph{$m$-fold tiling} if each $x \in \R^d$ is an element of exactly $m$ members of $\cT(K)$ unless $x$ is in the boundary of one of the members. Gruber and Lekkerkerker \cite[\S12.1]{MR893813} observed that, if the equality in \eqref{vdc:ineq} is attained, then $\cT(\frac{1}{2} C)$ is an $m$-fold tiling. This provides a connection to the theory of $m$-fold tilings by lattice translations of convex bodies. 

For $m=1$, the following converse implication is known to be true: If $\cT(K)$, with $K \in \cKo^d$, is a one-fold tiling, then $C= 2 K$ belongs to $\cKo^d(1)$ and attains  equality in \eqref{vdc:ineq}. Thus, studying the equality case of \eqref{vdc:ineq} for $m=1$ is equivalent to studying one-fold tilings by translations of a centrally symmetric convex body. This topic has a long history and, over the years, strong results on this topic have been discovered, both for general and concrete dimensions; see \cite[\S~12]{MR893813} and \cite[Ch.~32]{MR2335496}. One of the key results  is the Theorem of Venkov, Alexandrov and McMullen, which provides a  characterization of convex bodies that tile space by (lattice) translations; see \cite[\S32.2]{MR2335496}. 

In contrast to the case $m=1$, for larger values of $m \ge 2$, studying equality case of \eqref{vdc:ineq} is \emph{not} equivalent to studying arbitrary $m$-fold tilings by translations of a centrally symmetric convex body. In fact, if $\cT(K)$, with $K \in \cKo^d$, is an $m$-fold tiling, then $C=2K$ does not necessarily have $2m-1$ interior lattice points. This was observed by Gruber and Lekkerker \cite[\S12.1]{MR893813}. On the other hand, if $m \in \N$ is arbitrary, \emph{assuming} the property of having $2m-1$ interior lattice points, we get a characterization: for $C \in \cKo^d(2m-1)$, equality in \eqref{vdc:ineq} is attained if and only if $\cT(\frac{1}{2} C)$ is an $m$-fold tiling. Thus, studying the equality case in \eqref{vdc:ineq} for $m \ge 2$ can be reduced to studying special $m$-fold tilings $\cT(K)$ with $K \in \cKo^d$ and the property that $C= 2K$ belongs to $\cKo^d(2m-1)$. The latter observation explains the qualitative differences between the cases $m=1$ and $m \ge 2$. It should also be mentioned that the theory of $m$-fold lattice tilings for a general $m \in \N$ is more intricate than its classical case $m=1$ so that this theory does not immediately help to solve the problem we address in this paper. In particular, characterization of general $m$-fold lattice tilings by translations of convex bodies is a hard problem, even if the dimension is two and $m$ is fixed; see \cite{MR1383609,yang2017multiple,yang2017multiple:trans,zong2017characterization}. 

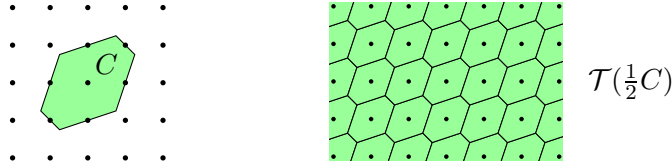
\begin{figure}
	\begin{center}
				\begin{tikzpicture}[scale=0.5]
		\filldraw[fill=myGreen] (1.25,0.75) -- (0.75,1.25) -- (-0.75,0.75) -- (-1.25,-0.75) -- (-0.75,-1.25) -- (0.75,-0.75) -- cycle;
		\foreach \x in {-2,...,2}
		\foreach \y in {-2,...,2}
		{
			\fill (\x,\y) circle (0.07);
		}
		\node at (0.5,0.5) {$C$};
		\end{tikzpicture}
		\hspace{5em}
		\begin{tikzpicture}[scale=0.25]
		\begin{scope}
		\clip (-6.2,-4.2) rectangle (6.2,4.2);
		\foreach \x in {-3,...,3}
		\foreach \y in {-3,...,3}
		{ 
			\filldraw[fill=myGreen] ({2*\x+1.25},{2*\y+0.75}) -- ({2*\x+0.75},{2*\y+1.25}) -- ({2*\x-0.75},{2*\y+0.75}) -- ({2*\x-1.25},{2*\y-0.75}) -- ({2*\x-0.75},{2*\y-1.25}) -- ({2*\x+0.75},{2*\y-0.75}) -- cycle;
			\fill (2*\x,2*\y) circle (0.12);
		}
		\end{scope}
		\node[right] at (7,0) {$\cT(\frac{1}{2} C)$};
		\end{tikzpicture}
	\end{center}
	\caption{\label{fig:ex:extremal} Example of an extremal convex body and the respective one-fold tiling}
\end{figure}

Following Gruber and Lekkerkerker \cite[\S12.1]{MR893813}, we call a set $C \in \cK_o^d(1)$  satisfying $\vol(C)=2^d$ \emph{extremal} (see also Fig.~\ref{fig:ex:extremal} for an example). Our main result is a characterization of the equality case in \eqref{vdc:ineq} for $d \ge 2$ and $m \ge 2$ in terms of extremal bodies. For $d=1$ the characterization is trivial: the segment $[-m,m]$ is the only volume maximizer in $\cK_o^1(2m-1)$. One can thus focus on dimensions $d \ge 2$. Given $d \ge 2$, a set $B \subseteq \R^{d-1}$ and functions $f, g : \R^{d-1} \to \R$, we introduce the set
\[
	L(B,f,g) := \setcond{(y,t) \in B \times \R}{f(y) \le t \le g(y)}.
\] 
A map $\phi : \R^d \to \R^d$ is called \emph{unimodular transformation} if $\phi$ is an affine transformation satisfying $\phi(\Z^d) = \Z^d$.  For $C \in \cKo^d$ and $B \in \cKo^{d-1}$, we say that $K$ is a \emph{cylindrical $m$-lifting} of $B$ if $C=\phi(L(B,a-m,a+m))$ for some linear unimodular transformation $\phi : \R^d \to \R^d$ and a linear function $a : \R^{d-1} \to \R$.

\begin{theorem}
	\label{thm:vdc:ineq:eqcase}
	Let $d, m \in \N$ and $d, m \ge 2$. Then a convex body $C \in \cK_o^d(2m-1)$ satisfies $\vol(C) = m 2^d$ if and only if $C$ is a cylindrical $m$-lifting of a $(d-1)$-dimensional extremal convex body. (See also Fig.~\ref{fig:CBl} for an illustration in dimension two.)
\end{theorem}

\begin{figure}
	\begin{center}
				\begin{tikzpicture}[scale=0.4]
		\filldraw[fill=myGreen] (1,2.5) -- (-1,1.5) -- (-1,-2.5) -- (1,-1.5) -- (1,1.5) -- cycle;
		\foreach \x in {-2,...,2}
		\foreach \y in {-3,...,3}
		{
			\filldraw (\x,\y) circle (0.07);
		}
		\filldraw (0,0)  circle (0.07); 
		\node[right] at (-0.1,0) {\scriptsize $o$};  
		\end{tikzpicture}
	\end{center}
	\caption{\label{fig:CBl} Example of $L(B,a-m,a+m)$ for $d=2, m=2$, $B=[-1,1]$ and $a(y) = \frac{1}{2} y$}
\end{figure}
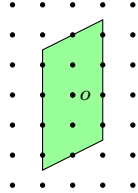

It is known that a $d$-dimensional extremal convex body is a polytope with at most $2^{d+1} -2$ facets (see \cite[Prop.~32.4, p.~470]{MR2335496} and \cite[\S12.3, Thm.~6]{MR893813}). This result and  Theorem~\ref{thm:vdc:ineq:eqcase} imply that convex bodies $C \in \cKo^{d-1}(2m-1)$ attaining equality $\vol(C) = m 2^d$ are prisms with at most $2^d$ facets. 

We present several further consequences of Theorem~\ref{thm:vdc:ineq:eqcase}.  For a family $\cT(K)$, with $K \in \cK^d$, one can look at members of $\cT(K)$ `colliding' with $K$. Formally, we say that $z \in \Z^d \setminus \{o\}$ is a \emph{collision vector} of $\cT(K)$ if the interiors of $K$ and $K+z$ have a non-empty intersection. The number of collision vectors is even, as they come in pairs $\pm z$. The set of collision vectors can be described as the set of all non-zero lattice vectors in the interior of $K-K$. The following corollary interprets van der Corput's inequality and Theorem~\ref{thm:vdc:ineq:eqcase} in the context of $m$-fold lattice tilings:

\begin{corollary}
	\label{cor:multiple:tilings}
	Let $\cT(K)$, with $K \in \cK^d$, be an $m$-fold lattice tiling with $m \ge 2$  having $2 N$ collision vectors. Then the following hold: 
	\begin{enumerate}[(a)]
		\item \label{cor:mult:ineq} One has $m \le N + 1$. 
		\item \label{cor:mult:eq:case} The equality $m= N + 1$ is attained if and only if a translation of $2 K$ is a cylindrical $m$-lifting of a $(d-1)$-dimensional extremal convex body. 
	\end{enumerate}
\end{corollary}

If, for an $m$-fold tiling $\cT(K)$ with $K \in \cK^d$, there exists a sub-lattice $\Lambda$ of $\Z^d$ such that $\setcond{K+z}{z \in \Lambda}$ is a one-fold tiling, then we say that $\cT(K)$ is a \emph{replication} of the one-fold tiling $\setcond{K+z}{z \in \Lambda}$. If this is the case, then $\cT(K)$ can be split into $m$ `copies' of $\setcond{K+z}{z \in \Lambda}$. In fact, since $\det(\Lambda)=m$, the quotient group $\Lambda / \Z^d$ has $m$ elements. We can thus choose $v_1,\ldots,v_m \in \Z^d$ with $\Z^d = \{v_1,\ldots,v_m\} + \Lambda$. With this choice, $\Z^d$ can be split into $m$ translations $v_i + \Lambda$ of $\Lambda$.  Correspondingly, $\cT(K)$ is split into $m$ one-fold tilings $\cT_i:=\setcond{K+z}{z \in v_i + \Lambda}$. 

The following corollary shows that $m$-fold tilings with a small number of collision vectors are replications of one-fold tilings.
\begin{corollary}
	\label{cor:N<=2}
	Let $m \in \N$ and $K \in \cK^d$ and let $\cT(K)$ be an  $m$-fold tiling with at most four collision vectors. Then $\cT(K)$ is a replication of a one-fold lattice tiling. (See also Fig.~\ref{fig:cor:N<=2} for an illustration.)
\end{corollary}

It would be interesting to determine the largest value $N^\ast \in \N$ with the property that every $m$-fold tiling $\cT(K)$, with $K \in \cK^d$, that has at most $2N^\ast$ collision vectors is a replication of a one-fold lattice tiling. According to Corollary~\ref{cor:N<=2}, one has $N^\ast \ge 2$. On the other hand, one can show $N^\ast < 10$ using the following example from \cite[\S1]{MR3063154}. Consider the octagon $K$ obtained as the convex hull of $\{0,1,2,3\}^2 \setminus \{0,3\}^2$. The family $\cT(K)$ is a seven-fold tiling with $20$ collision vectors; see also Fig.~\ref{fig:octagon}. It was mentioned above that every $d$-dimensional convex body tiling the space by translations is a polytope with at most $2^{d+1}-2$ facets. Thus, an octagon there exist no octagon tiling the plane by translations. This shows that $\cT(K)$ is not a replication of a one-fold tiling and implies $N^\ast < 10$. 

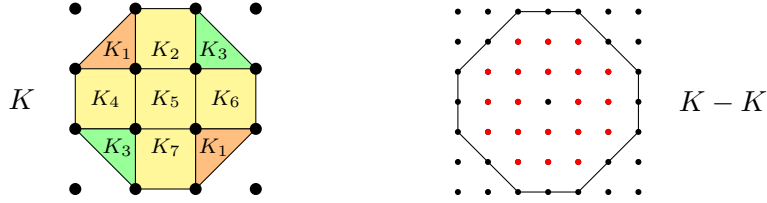
\begin{figure}[h!]
	\begin{center}
		\begin{tikzpicture}[scale=0.8]
	\filldraw[fill=orange!50!white] (0,2) -- (1,2) -- (1,3) -- cycle;
	\filldraw[fill=yellow!50!white] (1,2) -- (2,2) -- (2,3) -- (1,3) -- cycle;
	\filldraw[fill=green!40!white] (2,2) -- (3,2) -- (2,3) -- cycle;
	\filldraw[fill=yellow!50!white] (0,1) -- (1,1) -- (1,2) -- (0,2) -- cycle;
	\filldraw[fill=yellow!50!white] (1,1) -- (2,1) -- (2,2) -- (1,2) -- cycle;
	\filldraw[fill=yellow!50!white] (2,1) -- (3,1) -- (3,2) -- (2,2) -- cycle;
	\filldraw[fill=green!40!white] (1,0) -- (1,1) -- (0,1) -- cycle;
	\filldraw[fill=yellow!50!white] (1,0) -- (2,0) -- (2,1) -- (1,1) -- cycle;
	\filldraw[fill=orange!50!white] (2,0) -- (3,1) -- (2,1) -- cycle;
	\node[left] at (-0.5,1.5) {$K$};
	\node at (0.7,2.3) {\scriptsize $K_1$};
	\node at (1.5,2.3) {\scriptsize $K_2$};
	\node at (2.3,2.3) {\scriptsize $K_3$};
	\node at (0.5,1.5) {\scriptsize $K_4$};
	\node at (1.5,1.5) {\scriptsize $K_5$};
	\node at (2.5,1.5) {\scriptsize $K_6$};
	\node at (0.7,0.7) {\scriptsize $K_3$};
	\node at (1.5,0.7) {\scriptsize $K_7$};
	\node at (2.3,0.7) {\scriptsize $K_1$};
	\foreach \x in {0,1,2,3}
	\foreach \y in {0,1,2,3}
		\fill (\x,\y) circle (0.1);	
\end{tikzpicture}
\hspace{6em}
\begin{tikzpicture}[scale=0.4]
	\draw (3,1) -- (1,3) -- (-1,3) -- (-3,1) -- (-3,-1)  -- (-1,-3) -- (1,-3) -- (3,-1) -- cycle;
	\foreach \x in {-3,...,3}
	\foreach \y in {-3,...,3}
		\fill (\x,\y) circle (0.1);
	\foreach \x/\y in {1/0,2/0,2/1,1/1,1/2,0/2,0/1,-1/1,-1/2,-2/1}
	{
		\fill[red] (\x,\y) circle (0.1);
		\fill[red] ({-\x},{-\y}) circle (0.1);
	}
	\node[right] at (4,0) {$K-K$};
\end{tikzpicture}
	\end{center}
	\caption{\label{fig:octagon} For the octagon $K$ given as the convex hull of $\{-0,1,2,3\}^2 \setminus \{0,3\}^2$, the family $\cT(K)$ is a seven-fold tiling. This can be seen by decomposing $K$ into seven non-overlapping sets $K_1,\ldots,K_7$ with $\cT(K_i)$ being a one-fold tiling, for each $i \in \{1,\ldots,7\}$. Collision vectors can be descirbed as non-zero points in the interior of $K-K$. This  allows to check that $\cT(K)$ has $20$ collision vectors.}
\end{figure}

\begin{figure}[h!]
	\begin{center}
				\begin{tikzpicture}[scale=0.5]
		\filldraw[fill=myGreen] (-1/2,3/4) -- (-1/2,-5/4) -- (1/2,-3/4) -- (1/2,5/4) -- cycle;
		\node at (0,1.4) {$K$};
		
		\foreach \x in {4,8} 
		{
			\filldraw[fill=myGreen] ({-1/2+\x},3/4) -- ({-1/2+\x},-5/4) -- ({1/2+\x},-3/4) -- ({1/2+\x},5/4) -- cycle;
		}
		\foreach \x/\y in {4/1,8/-1}
		{
			\filldraw[fill=yellow,opacity=0.5] ({-1/2+\x},{3/4+\y}) -- ({-1/2+\x},{-5/4+\y}) -- ({1/2+\x},{-3/4+\y}) -- ({1/2+\x},{5/4+\y}) -- cycle;		
		}
		\foreach \x in {-2,...,18}
		\foreach \y in {-2,...,2}
		{
			\fill (\x,\y) circle (0.07);
		}
		\end{tikzpicture}
		\\
		\vspace{3ex}
		\begin{tikzpicture}[scale=0.5,baseline=2.5mm]
		\begin{scope}
		\clip (-2.5,-1.5) rectangle (2.5,2.5);
		\foreach \x in {-4,...,4}
		\foreach \y in {-4,-2,0,2,4}
		{
			\filldraw[fill=myGreen] ({-1/2+\x},{3/4+\y}) -- ({-1/2+\x},{-5/4+\y}) -- ({1/2+\x},{-3/4+\y}) -- ({1/2+\x},{5/4+\y}) -- cycle;	
			\fill (\x,\y) circle (0.1);
		}
		\foreach \x in {-4,...,4}
		\foreach \y in {-3,-1,1,3}
		{
			\draw[dotted] ({-1/2+\x},{3/4+\y}) -- ({-1/2+\x},{-5/4+\y}) -- ({1/2+\x},{-3/4+\y}) -- ({1/2+\x},{5/4+\y}) -- cycle;	
			\filldraw[fill=white] (\x,\y) circle (0.1);
		}
		\end{scope}
		\node at (0,-2) {$\cT(K)$};
		\end{tikzpicture}
		$\ = \ $		
		\begin{tikzpicture}[scale=0.5,baseline=2.5mm]
		\begin{scope}
		\clip (-2.5,-1.5) rectangle (2.5,2.5);
		\foreach \x in {-4,...,4}
		\foreach \y in {-4,-2,0,2,4}
		{
			\filldraw[fill=myGreen] ({-1/2+\x},{3/4+\y}) -- ({-1/2+\x},{-5/4+\y}) -- ({1/2+\x},{-3/4+\y}) -- ({1/2+\x},{5/4+\y}) -- cycle;	
			\fill (\x,\y) circle (0.1);
		}
		\end{scope}
		\node at (0,-2) {$\cT_1$};
		\end{tikzpicture}
		$\ \cup \ $
		\begin{tikzpicture}[scale=0.5,baseline=2.5mm]
		\begin{scope}
		\clip (-2.5,-1.5) rectangle (2.5,2.5);
		\foreach \x in {-4,...,4}
		\foreach \y in {-3,-1,1,3}
		{
			\filldraw[dotted,fill=myGreen] ({-1/2+\x},{3/4+\y}) -- ({-1/2+\x},{-5/4+\y}) -- ({1/2+\x},{-3/4+\y}) -- ({1/2+\x},{5/4+\y}) -- cycle;	
			\filldraw[fill=white] (\x,\y) circle (0.1);
		}
		\end{scope}
		\node at (0,-2) {$\cT_2$};
		\end{tikzpicture} \\
		\vspace{3ex}
		\begin{tikzpicture}[scale=0.5]
		\filldraw[fill=myGreen] (0,3/4) --(-1,1/4) --(-1,-1/4) -- (0,-3/4) -- (1,-1/4) -- (1,1/4) -- cycle;
		\node at (0,1.4) {$K$};
		\foreach \x in {4,8,12,16} 
		{
			\filldraw[fill=myGreen] ({0+\x},3/4) --({-1+\x},1/4) --({-1+\x},-1/4) -- ({0+\x},-3/4) -- ({1+\x},-1/4) -- ({1+\x},1/4) -- cycle;
		}
		\foreach \x/\y in {5/0,8/1,11/0,16/-1} 
		{
			\filldraw[fill=yellow,opacity=0.5] ({0+\x},{3/4+\y}) --({-1+\x},{1/4+\y}) --({-1+\x},{-1/4+\y}) -- ({0+\x},{-3/4+\y}) -- ({1+\x},{-1/4+\y}) -- ({1+\x},{1/4+\y}) -- cycle;
		}		\foreach \x in {-2,...,18}
		\foreach \y in {-2,...,2}
		{
			\fill (\x,\y) circle (0.07);
		}
		\end{tikzpicture}
		\\
		\vspace{3ex}
		\begin{tikzpicture}[scale=0.5,baseline=2.5mm]
		\begin{scope}
		\clip (-2.5,-1.5) rectangle (2.5,2.5);
		\foreach \x in {-3,...,3}
		\foreach \y in {-3,...,3}
		{
			\filldraw[fill=myGreen] ({0+\x+\y},{3/4+\x-\y}) --({-1+\x+\y},{1/4+\x-\y}) --({-1+\x+\y},{-1/4+\x-\y}) -- ({0+\x+\y},{-3/4+\x-\y}) -- ({1+\x+\y},{-1/4+\x-\y}) -- ({1+\x+\y},{1/4+\x-\y}) -- cycle;
			\fill ({\x+\y},{\x-\y}) circle (0.1);
		}
		\foreach \x in {-3,...,3}
		\foreach \y in {-3,...,3}
		{
			\draw[dotted] ({0+\x+\y},{3/4+\x-\y+1}) --({-1+\x+\y},{1/4+\x-\y+1}) --({-1+\x+\y},{-1/4+\x-\y+1}) -- ({0+\x+\y},{-3/4+\x-\y+1}) -- ({1+\x+\y},{-1/4+\x-\y+1}) -- ({1+\x+\y},{1/4+\x-\y+1}) -- cycle;
			\filldraw[fill=white] ({\x+\y+1},{\x-\y}) circle (0.1);
		}
		\end{scope}
		\node at (0,-2) {$\cT(K)$};
		\end{tikzpicture}
		$\ = \ $
		\begin{tikzpicture}[scale=0.5,baseline=2.5mm]
		\begin{scope}
		\clip (-2.5,-1.5) rectangle (2.5,2.5);
		\foreach \x in {-3,...,3}
		\foreach \y in {-3,...,3}
		{
			\filldraw[fill=myGreen] ({0+\x+\y},{3/4+\x-\y}) --({-1+\x+\y},{1/4+\x-\y}) --({-1+\x+\y},{-1/4+\x-\y}) -- ({0+\x+\y},{-3/4+\x-\y}) -- ({1+\x+\y},{-1/4+\x-\y}) -- ({1+\x+\y},{1/4+\x-\y}) -- cycle;
			\fill ({\x+\y},{\x-\y}) circle (0.1);
		}
		\end{scope}
		\node at (0,-2) {$\cT_1$};
		\end{tikzpicture}
		$\ \cup \ $
		\begin{tikzpicture}[scale=0.5,baseline=2.5mm]
		\begin{scope}
		\clip (-2.5,-1.5) rectangle (2.5,2.5);
		\foreach \x in {-3,...,3}
		\foreach \y in {-3,...,3}
		{
			\filldraw[dotted,fill=myGreen] ({0+\x+\y},{3/4+\x-\y+1}) --({-1+\x+\y},{1/4+\x-\y+1}) --({-1+\x+\y},{-1/4+\x-\y+1}) -- ({0+\x+\y},{-3/4+\x-\y+1}) -- ({1+\x+\y},{-1/4+\x-\y+1}) -- ({1+\x+\y},{1/4+\x-\y+1}) -- cycle;
		}
		\foreach \x in {-4,...,4}
		\foreach \y in {-4,...,4}
		{
			\filldraw[fill=white] ({\x+\y+1},{\x-\y}) circle (0.1);
		}
		\end{scope}
		\node at (0,-2) {$\cT_2$};
		\end{tikzpicture} \\
		\vspace{3ex}
		\begin{tikzpicture}[scale=0.5]
		\filldraw[fill=myGreen] (1/4,3/4) --(-5/4,1/4) --(-1/4,-3/4) -- (5/4,-1/4) -- cycle;
		\node at (0,1.4) {$K$};
		\foreach \x in {4,8,12,16} 
		{
			\filldraw[fill=myGreen] ({1/4+\x},3/4) --({-5/4+\x},1/4) --({-1/4+\x},-3/4) -- ({5/4+\x},-1/4) -- cycle;
		}
		\foreach \x/\y in {5/0,8/1,11/0,16/-1} 
		{
			\filldraw[fill=yellow,opacity=0.5] ({1/4+\x},{3/4+\y}) --({-5/4+\x},{1/4+\y}) --({-1/4+\x},{-3/4+\y}) -- ({5/4+\x},{-1/4+\y}) -- cycle;
		}		
		\foreach \x in {-2,...,18}
		\foreach \y in {-2,...,2}
		{
			\fill (\x,\y) circle (0.07);
		}
		\end{tikzpicture}
		\\
		\vspace{3ex}
		\begin{tikzpicture}[scale=0.5,baseline=2.5mm]
		\begin{scope}
		\clip (-2.5,-1.5) rectangle (2.5,2.5);
		\foreach \x in {-4,...,4}
		\foreach \y in {-4,...,4}
		{
			\filldraw[fill=myGreen] ({1/4+\x+\y},{3/4+\x-\y}) --({-5/4+\x+\y},{1/4+\x-\y}) --({-1/4+\x+\y},{-3/4+\x-\y}) -- ({5/4+\x+\y},{-1/4+\x-\y}) -- cycle;
			\fill ({\x+\y},{\x-\y}) circle (0.1);
		}
		\foreach \x in {-4,...,4}
		\foreach \y in {-4,...,4}
		{
			\draw[dotted] ({1/4+\x+\y+1},{3/4+\x-\y}) --({-5/4+\x+\y+1},{1/4+\x-\y}) --({-1/4+\x+\y+1},{-3/4+\x-\y}) -- ({5/4+\x+\y+1},{-1/4+\x-\y}) -- cycle;
			\filldraw[fill=white] ({\x+\y+1},{\x-\y}) circle (0.1);
		}
		\end{scope}
		\node at (0,-2) {$\cT(K)$};
		\end{tikzpicture}
		$\ = \ $
		\begin{tikzpicture}[scale=0.5,baseline=2.5mm]
		\begin{scope}
		\clip (-2.5,-1.5) rectangle (2.5,2.5);
		\foreach \x in {-4,...,4}
		\foreach \y in {-4,...,4}
		{
			\filldraw[fill=myGreen] ({1/4+\x+\y},{3/4+\x-\y}) --({-5/4+\x+\y},{1/4+\x-\y}) --({-1/4+\x+\y},{-3/4+\x-\y}) -- ({5/4+\x+\y},{-1/4+\x-\y}) -- cycle;
			\fill ({\x+\y},{\x-\y}) circle (0.1);
		}
		\end{scope}
		\node at (0,-2) {$\cT_1$};
		\end{tikzpicture}
		$\ \cup \ $
		\begin{tikzpicture}[scale=0.5,baseline=2.5mm]
		\begin{scope}
		\clip (-2.5,-1.5) rectangle (2.5,2.5);
		\foreach \x in {-4,...,4}
		\foreach \y in {-4,...,4}
		{
			\filldraw[dotted,fill=myGreen] ({1/4+\x+\y+1},{3/4+\x-\y}) --({-5/4+\x+\y+1},{1/4+\x-\y}) --({-1/4+\x+\y+1},{-3/4+\x-\y}) -- ({5/4+\x+\y+1},{-1/4+\x-\y}) -- cycle;
			\filldraw[fill=white] ({\x+\y+1},{\x-\y}) circle (0.1);
		}
		\end{scope}
		\node at (0,-2) {$\cT_2$};
		\end{tikzpicture}
	\end{center}
	\caption{\label{fig:cor:N<=2} Examples in dimension two, illustrating Corollary~\ref{cor:N<=2}}
\end{figure}
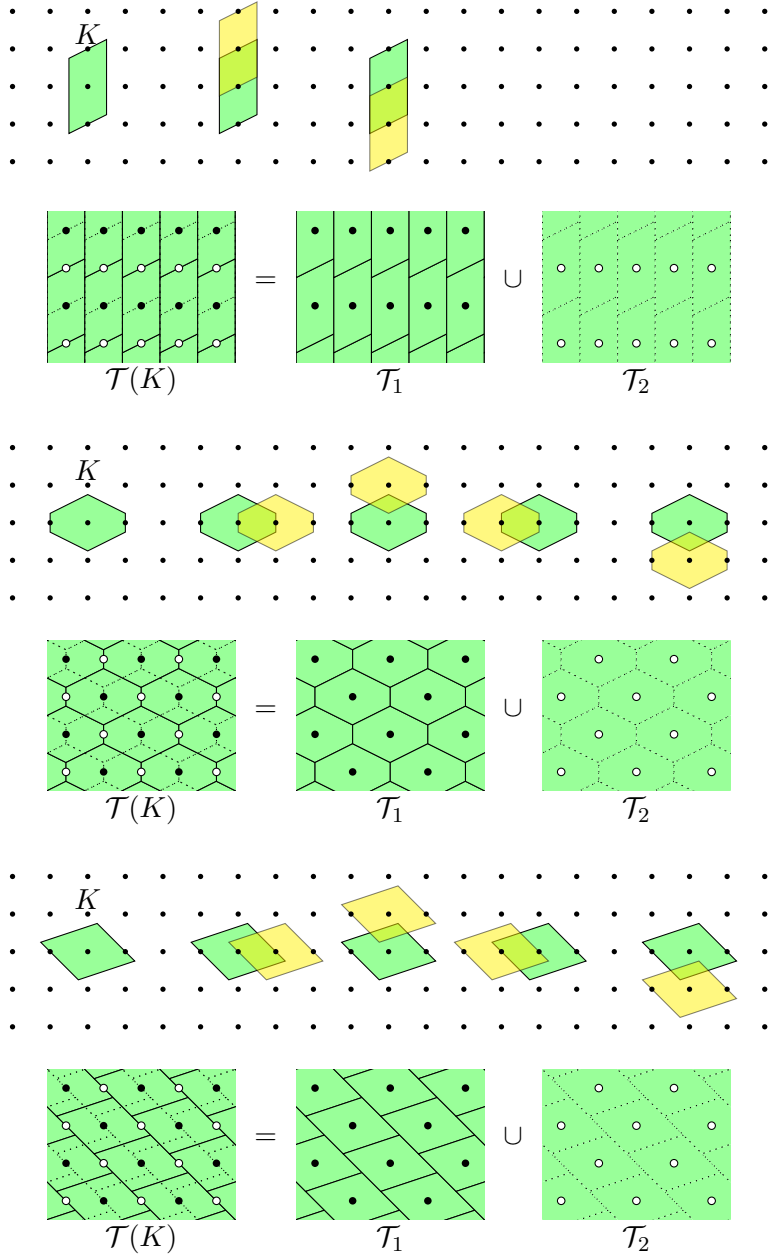

The paper is organized as follows. Section~\ref{sect:background} contains notation and terminology along with a few basic observations.  In Section~\ref{sec:basic}, we revise the proof of van der Corput's inequality in order to add some necessary refinements. Section~\ref{sec:eq:case} contains the proof of Theorem~\ref{thm:vdc:ineq:eqcase} and its consequences.

\section{Preliminaries}

\label{sect:background}

Let $\N:=\{1,2,3,\ldots\}$ be the set of natural numbers and let $d \in \N$. The cardinality of a set $X$ is denoted by $|X|$. We define the \emph{dimension} $\dim(X)$ of $X \subseteq \R^d$ as the dimension of the affine hull of $X$. Let $o$ denote the origin and $e_1,\ldots,e_d$ the standard basis of $\R^d$. A set $X \subseteq \R^d$ is called \emph{$o$-symmetric} if, for every $x \in X$, the point $-x$ also belongs to $X$. The interior and closure of $X \subseteq \R^d$ are denoted by $\intr{X}$ and $\cl{X}$, respectively.  

For $A \subseteq \R^d$, we denote by $\chf_A : \R^d \to \R$ the characteristic function of $A$, which is given by 
\[
	\chf_A(x) :=
	\begin{cases}
	  1 & \text{for} \ x \in A, 
	\\ 0 & \text{for} \ x \in \R^d \setminus A.
	\end{cases}
\]  For $X, Y \subseteq \R^d$ and $\alpha \in \R$,  we use the notation 
\begin{align*}
	X+Y & := \setcond{x+y}{x \in X, \ y \in Y},
	\\ X-Y & := \setcond{x-y}{x \in X, \ y \in Y}, 
	\\ \alpha X & := \setcond{\alpha x}{x \in X}.
\end{align*}

A set $X \subseteq \R^d$ is called an \emph{arithmetic progression} if, for some $k \in \N$, the set $X$ is the image of $\{1,\ldots,k\}$ under an affine transformation $\phi :\R \to \R^d$. By $\vol$ we denote the volume, that is, the Lebesgue measure on $\R^d$, scaled so that $\vol([0,1]^d)=1$. A set $K \subseteq \R^d$ is called a \emph{convex body} if $K$ is a compact convex set with non-empty interior. By $\cK^d$ and $\cKo^d$ we denote the family of all convex bodies in $\R^d$ and all $o$-symmetric convex bodies in $\R^d$, respectively. For basic information on convex sets and convex polytopes we refer to \cite{MR3155183,MR2335496}. Observe that one has 
\begin{equation}
	\label{cl:int}
	\cl{\intr{K}}=K \qquad \text{for every} \ K \in \cK^d
\end{equation}
and 
\begin{equation}
	\label{int:sum}
	\intr{A+B} = \intr{A} + B = A + \intr{B} = \intr{A}+\intr{B} \qquad \text{for all} \ A, B \in \cK^d.
\end{equation}

For $d \ge 2$, we will use the projection $\pi : \R^d \to \R^{d-1}$ onto the first $d-1$ components:
\[
	\pi(x_1,\ldots,x_d) := (x_1,\ldots,x_{d-1}).
\]
For $K \in \cK^d$ and $y \in \pi(K)$ we introduce
\begin{align*}
I_K(y) & := \setcond{t \in \R}{(y,t) \in K},
\\ f_K(y) & := \text{length of} \ I_K(y).
\end{align*}
In other words, $\{y \} \times I_K(y)$ is the intersection of $K$ and the vertical line $\{y \} \times \R$ and $f_K(y)$ is the length of this intersection.

For $m \in \N$, a family $\cT$ of subsets of $\R^d$ is called: an \emph{$m$-fold packing} if each $x \in \R^d$ is in at most $m$ sets of the family $\setcond{\intr{K}}{K \in \cT}$, an \emph{$m$-fold covering} if each $x \in \R^d$ is in at least $m$ sets of the family $\cT$, and 
an \emph{$m$-fold tiling} if $\cT$ is both an $m$-fold packing and an $m$-fold covering. If $\cT$ is an $m$-fold tiling for some choice of $m \in \N$, we say that $\cT$ is a \emph{multiple tiling} and call $m$ the \emph{multiplicity} of $\cT$. 

We refer to the elements of $\Z^d$ as \emph{lattice vectors} or \emph{lattice points}. For $A \subseteq \R^d$, we introduce the family 
\[
\cT(A) := \setcond{A +z}{z \in \Z^d}
\]
of all translations of $A$ by lattice vectors. With $\cT(A)$ we associate the \emph{multiplicity function} $\mult(A,\dotvar) : \R^d \to \R$ that counts how many elements of $\cT(A)$ contain a given point $x \in \R^d$. Formally, $\mult(A,x)$ can be expressed as follows:
\begin{align}
	\mult(A,x) :=& \left|\setcond{z \in \Z^d}{x \in A + z}\right| \label{mult:a}
	\\ =&  |A \cap (x+ \Z^d)| \label{mult:b}
	\\ =& \sum_{z \in \Z^d} \chf_A(x+z). \label{mult:c}
\end{align}
It is clear that $\mult(A,x)$ is $\Z^d$-periodic in $x$.  For $K \in \cK^d$, the family $\cT(K)$ is an $m$-fold covering if and only if $\mult(K,x) \ge m$ for all $x \in \R^d$ and $\cT(K)$ is an $m$-fold packing if and only if $\mult(\intr{K},x) \le m$ for all $x \in \R^d$. We say that $z \in \Z^d \setminus \{o\}$ is a \emph{collision vector} of $\cT(K)$ if $\intr{K} \cap (\intr{K}+z) \ne \emptyset$. Collision vectors describe which pairs of the family $\cT(K)$ overlap. In view of \eqref{int:sum}, the set of all collision vectors of $\cT(K)$ can be represented as $\intr{K-K} \cap \Z^d \setminus \{o\}$; see also Fig.~\ref{fig:col:vec} for an illustration. 

\begin{figure}
	\begin{center}
				\begin{tikzpicture}[scale=0.6]
		\filldraw[fill=myGreen] (0,0) -- (1,1) -- (1,2) -- (0,2) -- (-1,1) -- cycle;
		\draw (0,0) -- (-1,-1) -- (-1,-2) -- (0,-2) -- (1,-1) -- cycle;
		\foreach \x in {-2,...,2}
		\foreach \y in {-2,...,2}
		\fill (\x,\y) circle (0.07);
		\node at (0.4,0) {\scriptsize{$o$}};		
		\node at (0.4,1.5) {\scriptsize{$K$}};
		\node at (-0.4,-1.5) {\scriptsize{$-K$}};
		\end{tikzpicture}
		\hspace{4em}		
		\begin{tikzpicture}[scale=0.6]
		\draw (1,2) -- (0,2) -- (-2,0) -- (-2,-1) -- (-1,-2) -- (0,-2) -- (2,0) -- (2,1) -- cycle;
		\foreach \x in {-2,...,2}
		\foreach \y in {-2,...,2}
		\fill (\x,\y) circle (0.07);
		\draw[->] (0,0) -- (1,0);
		\draw[->] (0,0) -- (1,1);
		\draw[->] (0,0) -- (0,1);
		\draw[->] (0,0) -- (-1,0);
		\draw[->] (0,0) -- (-1,-1);
		\draw[->] (0,0) -- (0,-1);
		\node at (0.5,1.4) {\scriptsize{$K-K$}};
		%
		\end{tikzpicture}
		\\
		\vspace{5ex}
		\begin{tikzpicture}[scale=0.4]
		\filldraw[fill=myGreen] (-1,0) -- (0,-1) -- (1,0) -- (1,1) -- (0,1) -- cycle;
		\foreach \x in {-2,...,2}
		\foreach \y in {-2,...,2}
		\fill (\x,\y) circle (0.07);
		\def\u{1}
		\def\v{0}
		\filldraw[fill=yellow,opacity=0.5] ({\u+(-1)},{\v+0}) -- ({\u+0},{\v-1}) -- ({\u+1},{\v+0}) -- ({\u+1},{\v+1}) -- ({\u+0},{\v+1}) -- cycle;
		\draw[->] (-2,2) -- (-1,2);
		\end{tikzpicture}
		\hspace{1.5em}
		\begin{tikzpicture}[scale=0.4]
		\filldraw[fill=myGreen] (-1,0) -- (0,-1) -- (1,0) -- (1,1) -- (0,1) -- cycle;
		\foreach \x in {-2,...,2}
		\foreach \y in {-2,...,2}
		\fill (\x,\y) circle (0.07);
		\def\u{1}
		\def\v{1}
		\filldraw[fill=yellow,opacity=0.5]({\u+(-1)},{\v+0}) -- ({\u+0},{\v-1}) -- ({\u+1},{\v+0}) -- ({\u+1},{\v+1}) -- ({\u+0},{\v+1}) -- cycle;
		\draw[->] (-2,1) -- (-1,2);
		\end{tikzpicture}		
		\hspace{1.5em}
		\begin{tikzpicture}[scale=0.4]
		\filldraw[fill=myGreen] (-1,0) -- (0,-1) -- (1,0) -- (1,1) -- (0,1) -- cycle;
		\foreach \x in {-2,...,2}
		\foreach \y in {-2,...,2}
		\fill (\x,\y) circle (0.07);
		\def\u{0}
		\def\v{1}
		\filldraw[fill=yellow,opacity=0.5]({\u+(-1)},{\v+0}) -- ({\u+0},{\v-1}) -- ({\u+1},{\v+0}) -- ({\u+1},{\v+1}) -- ({\u+0},{\v+1}) -- cycle;
		\draw[->] (-2,1) -- (-2,2);
		\end{tikzpicture}			
		\hspace{1.5em}
		\begin{tikzpicture}[scale=0.4]
		\filldraw[fill=myGreen] (-1,0) -- (0,-1) -- (1,0) -- (1,1) -- (0,1) -- cycle;
		\foreach \x in {-2,...,2}
		\foreach \y in {-2,...,2}
		\fill (\x,\y) circle (0.07);
		\def\u{-1}
		\def\v{0}
		\filldraw[fill=yellow,opacity=0.5]({\u+(-1)},{\v+0}) -- ({\u+0},{\v-1}) -- ({\u+1},{\v+0}) -- ({\u+1},{\v+1}) -- ({\u+0},{\v+1}) -- cycle;
		\draw[->] (-1,2) -- (-2,2);
		\end{tikzpicture}			
		\hspace{1.5em}
		\begin{tikzpicture}[scale=0.4]
		\filldraw[fill=myGreen] (-1,0) -- (0,-1) -- (1,0) -- (1,1) -- (0,1) -- cycle;
		\foreach \x in {-2,...,2}
		\foreach \y in {-2,...,2}
		\fill (\x,\y) circle (0.07);
		\def\u{-1}
		\def\v{-1}
		\filldraw[fill=yellow,opacity=0.5]({\u+(-1)},{\v+0}) -- ({\u+0},{\v-1}) -- ({\u+1},{\v+0}) -- ({\u+1},{\v+1}) -- ({\u+0},{\v+1}) -- cycle;
		\draw[->] (-1,2) -- (-2,1);
		\end{tikzpicture}			
		\hspace{1.5em}
		\begin{tikzpicture}[scale=0.4]
		\filldraw[fill=myGreen] (-1,0) -- (0,-1) -- (1,0) -- (1,1) -- (0,1) -- cycle;
		\foreach \x in {-2,...,2}
		\foreach \y in {-2,...,2}
		\fill (\x,\y) circle (0.07);
		\def\u{0}
		\def\v{-1}
		\filldraw[fill=yellow,opacity=0.5]({\u+(-1)},{\v+0}) -- ({\u+0},{\v-1}) -- ({\u+1},{\v+0}) -- ({\u+1},{\v+1}) -- ({\u+0},{\v+1}) -- cycle;
		\draw[->] (-2,2) -- (-2,1);
		\end{tikzpicture}					
	\end{center}
	\caption{\label{fig:col:vec} For the pentagon $K$ with the vertices $(0,0),(1,1), (1,2), (0,2), (-1,1)$, the family $\cT(K)$ has six collision vectors $\pm (1,0), \pm (1,1), \pm (0,1)$. The figure depicts the collision vectors as non-zero lattice vectors in $\intr{K-K}$ and illustrates the respective collisions among members of $\cT(K)$.}
\end{figure}
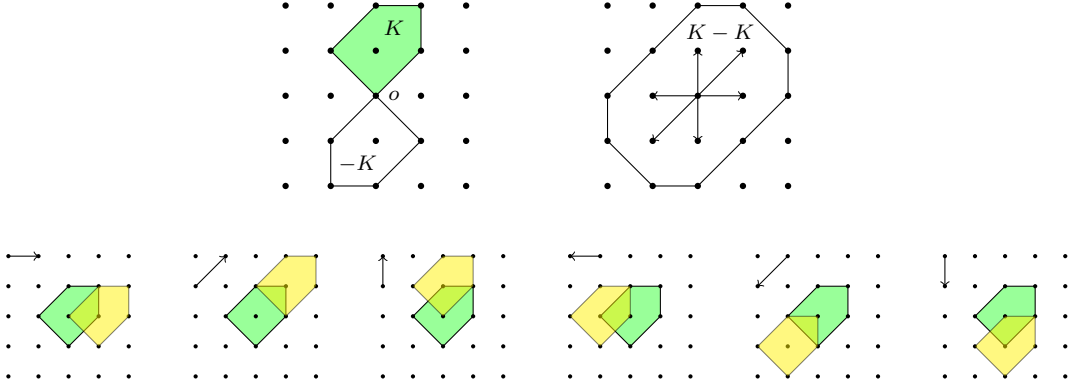

\section{Van der Corput's inequality and refinements}

\label{sec:basic}

This section presents several basic results from the geometry of numbers in a revised and refined form. While the content of this section is mostly not new, the presentation and proofs are somewhat different from the standard reference books \cite{MR893813} and \cite{MR1434478}. For reader's convenience, we give a self-contained presentation. While several of the presented results are known to hold for sets which are not necessarily convex, we prefer to keep the focus on the family of convex sets. 

We use the approach of Uhrin \cite{MR612682}, who showed that van der Corput's inequality can be deduced from bounds on the cardinality of the difference set: 

\begin{theorem}[Difference-set inequality and its equality case; {\cite[(2.4)]{MR1058212}}]
	\label{thm:dsi}
	Let $X \subseteq \R^d$ be a non-empty finite set. Then
	\(
		|X - X| \ge 2 |X| -1, 
	\)
	 and the equality $|X-X| = 2 |X| -1$ is attained if and only if $X$ is an arithmetic progression. 
\end{theorem}

\begin{lemma} \label{lem:wrapper}
	Let $A \subseteq \R^d$ be a $d$-dimensional bounded convex set. Then
	\[
		\vol(A) = \int\limits_{[0,1]^d} \mult(A,x) \dd x.
	\]
\end{lemma}
\begin{proof}
	\begin{align*}
	\vol(A) & = \int_{\R^d} \chf_A(x) \dd x 
	\\ & = \sum_{z \in \Z^d} \int_{[0,1]^d} \chf_A(x+z) \dd x & & \text{(by tiling $\R^d$ into translates of $[0,1]^d$)}
	\\ & =  \int_{[0,1]^d} \sum_{z \in \Z^d} \chf_A(x+z) \dd x & & \text{(exchanging the order  $\sum$ and $\int$)}
	\\ & = \int_{[0,1]^d} \mult(A,x) \dd x & & \text{(using \eqref{mult:c}).}
	\end{align*}
\end{proof}

\begin{theorem}[On collision vectors and $m$-fold packings]
	\label{coll:mult:packings}
	Let $K \in \cK^d$.
	Consider the set
	\(
		Z:= \intr{K-K} \cap \Z^d \setminus \{o\}
	\)
	of the collision vectors of the family $\cT(K)$. Then $\cT(K)$ is an $m$-fold packing with $m= \frac{1}{2} |Z|+1$, when $Z$ is arbitrary, and $m=\frac{1}{2} |Z|$, when $\dim(Z) \ge 2$. 
\end{theorem}
\begin{proof}
	Let $A:= \intr{K}$. For $x \in \R^d$, we introduce the set $A_x:= A \cap (x+\Z^d)$. In view of \eqref{mult:b}, one has $\mult(A,x)=|A_x|$. Since $A_x$ is a subset of both $A$ and $x+\Z^d$, the difference set $A_x-A_x$ is a subset of $A-A = \intr{K-K}$ and $(x+\Z^d) - (x+\Z^d) = \Z^d$. This yields the inclusion
	\begin{align}
	A_x - A_x  & \subseteq Z \cup \{o\}.\label{diff:Ax:eq}
	\end{align}
	Hence
	\begin{align*}
		\mult(A,x) & =|A_x|  & & \text{(by \eqref{mult:b})}
		\\ & \le \frac{1}{2} (|A_x-A_x|+1) & & \text{(by Theorem~\ref{thm:dsi})}
			\\ & \le \frac{1}{2} |Z|+1 & & \text{(by \eqref{diff:Ax:eq})}.
	\end{align*} 
	We have verified the assertion for an arbitrary $Z$. In the case $\dim(Z) \ge 2$, we need to check the stronger inequality $\mult(A,x) \le \frac{1}{2} |Z|$. If $A_x = \emptyset$, the latter inequality holds. If $A_x$ is non-empty and is not an arithmetic progression, Theorem~\ref{thm:dsi} yields the strict inequality 
	\(
	|A_x - A_x| > 2 |A_x| - 1.
	\)
	Both the left and the right hand side of the latter strict inequality are odd numbers. Thus, the inequality can be reformulated as 
	\(
		|A_x  - A_x| \ge 2 |A_x| + 1.
	\)
	This implies 
	\begin{align*}
		\mult(A,x) = |A_x| & \le \frac{1}{2}(|A_x-A_x| -1) 
		 \le \frac{1}{2} |Z|.
	\end{align*}
	In the case when $A_x$ is an arithmetic progression, we have $\dim(A_x-A_x)=1$. 
	Since both $A_x-A_x$ and $Z$ are $o$-symmetric with $\dim(A_x-A_x)=1$ and $\dim(Z) \ge 2$, we conclude that $Z$ contains a pair of points symmetric with respect to the origin that are not in $A_x-A_x$. Furthermore, since $A_x$ is an arithmetic progression, we have the equality 
	\(
	|A_x| = \frac{1}{2}(|A_x-A_x|+1).
	\) The above observations imply
	\[
		\mult(A,x) = |A_x| = \frac{1}{2}(|A_x-A_x| + 1) \le \frac{1}{2} |Z|.
	\]
\end{proof}

\begin{proposition}
	\label{multiple:pack:cov:vol} 
	Let $m \in \N$ and $K \in \cK^d$. Then the following hold:
		\begin{enumerate}[(a)]
			\item \label{pack:vol} $\vol(K) \le m$ if $\cT(K)$ is an $m$-fold packing.
			\item \label{cover:vol} $\vol(K) \ge m$ if $\cT(K)$ is an $m$-fold covering.
			\item \label{tiling:vol} $\vol(K)=m$ if $\cT(K)$ is an $m$-fold tiling. 
		\end{enumerate}
\end{proposition}
\begin{proof}
		\eqref{pack:vol}: If $\cT(K)$ is an $m$-fold packing, then using Lemma~\ref{lem:wrapper} we get 
		\[
		\vol(K) = \vol(\intr{K}) = \int_{[0,1]^d} \underbrace{\mult(\intr{K},x)}_{\le m} \dd x \le m.
		\]
		\eqref{cover:vol}: Analogously, if $\cT(K)$ is an $m$-fold covering, then using Lemma~\ref{lem:wrapper} we get 
		\[
		\vol(K) = \int_{[0,1]^d} \underbrace{\mult(K,x)}_{\ge m} \dd x \ge m.
		\]
		Assertion \eqref{tiling:vol} is a direct consequence of \eqref{pack:vol} and \eqref{cover:vol}.
\end{proof}

Theorem~\ref{thm:tiling:char} below is a simple characterization of $m$-fold tilings. In the proof of the characterization, we use the following lemma.

\begin{lemma}
	\label{lem:semicont}
	Let $A \subseteq \R^d$ be bounded. Then the following hold: 
	\begin{enumerate}[(a)]
		\item \label{open:semicont} If $A$ is open, then $\mult(A,x) \ge \mult(A,x^\ast)$ holds for all $x$ in a neighborhood of $x^\ast$.
		\item \label{closed:semicont} If $A$ is closed, then $\mult(A,x) \le \mult(A,x^\ast)$ holds for all $x$ in a neighborhood of $x^\ast$. 
	\end{enumerate}
\end{lemma}
\begin{proof}
	Since $A$ is bounded, only finitely many  members of $\cT(A)$ contain $x^\ast$. Hence, if $A$ is open, then $U := \bigcap \setcond{A+z}{z \in \Z^d, \ x^\ast \in A+z}$ is an open set containing $x^\ast$. By construction, $\mult(A,x) \ge \mult(A,x^\ast)$ holds for every $x \in U$. Thus, \eqref{open:semicont} is true.
	
	Assume now that $A$ is closed. We fix an arbitrary open bounded neighborhood $W$ of $x^\ast$. Since $A$ is bounded, only finitely many members of $\cT(A)$ have a non-empty intersection with $W$. Hence $U:=W \setminus \bigcup \setcond{A+z}{z \in \Z^d, x^\ast \not\in A+z}$ is an open neighborhood of $x^\ast$. By construction, $\mult(A,x) \le \mult(A,x^\ast)$ holds for all $x \in U$. This shows that \eqref{closed:semicont} is true.
\end{proof}

\begin{theorem}
	\label{thm:tiling:char}
	Let $m \in \N$ and $K \in \cK^d$. Then the following conditions are equivalent: 
	\begin{enumerate}[(i)]
		\item \label{pack:vol=m} $\cT(K)$ is an $m$-fold packing with $\vol(K)=m$.
		\item \label{cov:vol=m} $\cT(K)$ is an $m$-fold covering with $\vol(K) =m$.
		\item \label{tiling} $\cT(K)$ is an $m$-fold tiling.
	\end{enumerate} 
\end{theorem}
\begin{proof}
	\eqref{pack:vol=m} $\Rightarrow$ \eqref{cov:vol=m}: Assume \eqref{pack:vol=m} is true. If $\cT(K)$ is not an $m$-fold covering, then there exists $x^\ast$ such that $\mult(K,x^\ast) \le m-1$. By Lemma~\ref{lem:semicont}\eqref{closed:semicont}, $\mult(K,x) \le m-1$ holds for all $x$ in an open neighborhood $U$ of $x^\ast$. We fix $z \in \Z^d$ such that $[0,1]^d+z$ and $U$ have a non-empty intersection. Fix a non-empty  open subset $W \subseteq ([0,1]^d + z) \cap U$. In view of the $\Z^d$-periodicity of $\mult(K,\dotvar)$, we have $\mult(K,x) \le m-1$ for all $x \in W$. Hence
	\begin{align*}
	\vol(K)  & = \int_{[0,1]^d} \mult(\intr{K},x) \dd x 	& & \text{(by Lemma~\ref{lem:wrapper})}
	\\ & = \int_{W} \underbrace{\mult(\intr{K},x)}_{\le m-1} \dd x + \int_{[0,1]^d \setminus W} \underbrace{\mult(\intr{K},x)}_{\le m} \dd x & & \text{(decomposing $[0,1]^d$)}
	\\	& \le (m-1) \vol(W) + m (1 - \vol(W)) & & 
	\\ & = m - \vol(W).
	\end{align*}
	We obtain the inequality $\vol(K) < m$, contradicting $\vol(K)=m$. This shows that \eqref{pack:vol=m} implies \eqref{cov:vol=m}.
		
	\eqref{cov:vol=m} $\Rightarrow$ \eqref{pack:vol=m}:	Assume \eqref{cov:vol=m} is true. Let $A:=\intr{K}$. If $\cT(K)$ is not an $m$-fold packing, then there exists $x^\ast$ with $\mult(A,x^\ast) \ge m+1$. By Lemma~\ref{lem:semicont}\eqref{open:semicont}, $\mult(A,x) \ge m+1$ holds for all $x$ in an open neighborhood $U$ of $x^\ast$. Since the function $\mult(A,\dotvar)$ is $\Z^d$-periodic, analogously to the proof of the implication \eqref{pack:vol=m} $\Rightarrow$ \eqref{cov:vol=m}, one can fix a non-empty open set $W \subseteq [0,1]^d$ with $\mult(A,x) \ge m+1$ of every $x \in W$. We get
	\begin{align*}
	\vol(K) 
	 & = \int_{[0,1]^d} \mult(K,x) \dd x & & \text{(by Lemma~\ref{lem:wrapper})}
	\\ & = \int_{W} \underbrace{\mult(K,x)}_{\ge m+1} \dd x + \int_{[0,1]^d \setminus W} \underbrace{\mult(K,x)}_{\ge m} \dd x & & \text{(decomposing $[0,1]^d$)}
	\\ & \ge (m+1) \vol(W) + m (1-\vol(W)) & & 
	\\ & = \vol(W) + m.
	\end{align*}
	We obtain $\vol(K) > m$, contradicting $\vol(K)=m$. Thus, \eqref{cov:vol=m} implies \eqref{pack:vol=m}.
	
	It remains to check the equivalence of \eqref{tiling} and the other two conditions. By Proposition~\ref{multiple:pack:cov:vol}\eqref{tiling:vol}, condition \eqref{tiling} implies both \eqref{pack:vol=m} and \eqref{cov:vol=m}. Furthermore, if \eqref{pack:vol=m} or \eqref{cov:vol=m} is true, then due to their equivalence, both of them are true. But then \eqref{tiling} is also true.
\end{proof}

\begin{theorem}[Van der Corput's inequality and refinements]
	\label{thm:vdc:ineq}
	Let $m \in \N$ and $C \in \cKo^d(2m-1)$. Then the following hold:
	\begin{enumerate}[(a)]
		\item \label{vdc} 
		\(
			\vol(C) \le m 2^d.
		\)
		\item \label{vdc:modified} $\vol(C) \le (m-1) 2^d$, if $\dim(\Z^d \cap \intr{C}) \ge 2$.
		\item \label{vdc:eq:tiling} $\vol(C) = m 2^d$ holds if and only if $\cT(\frac{1}{2} C)$ is an $m$-fold tiling.
	\end{enumerate} 
\end{theorem}
\begin{proof}
	The collision vectors of the family  $\cT(\frac{1}{2} C)$ are the vectors in $\intr{C} \cap \Z^d \setminus \{o\}$. Hence, $\cT(\frac{1}{2} C)$ has $2(m-1)$ collision vectors. Assertion \eqref{vdc} and \eqref{vdc:modified} follow by applying Theorem~\ref{coll:mult:packings} and then Proposition~\ref{multiple:pack:cov:vol}\eqref{pack:vol} to the family $\cT(\frac{1}{2}C)$. 

	It remains to verify \eqref{vdc:eq:tiling}. If $\cT(\frac{1}{2}C)$ is an $m$-fold tiling, then $\vol(C) = m 2^d$ follows directly from Proposition~\ref{multiple:pack:cov:vol}\eqref{tiling:vol}. Conversely, if $\vol(C) = m 2^d$, then the equivalence \eqref{pack:vol=m} $\Leftrightarrow$ \eqref{tiling} of Theorem~\ref{thm:tiling:char} applied to the family $\cT(\frac{1}{2} C)$ yields that $\cT(\frac{1}{2} C)$ is an $m$-fold tiling.
\end{proof}

\section{Proofs of Theorem~\ref{thm:vdc:ineq:eqcase} and its consequences}

\label{sec:eq:case}

\begin{lemma}
	\label{lem:cover:xray}
	Let $m \in \N$ and $m \ge 2$ and let $\cT(K)$, with $K \in \cK^d$ , be an $m$-fold covering with the property that all collision vectors are multiples of $e_d$. Then $f_K(y) \ge m$ holds for every $y \in \intr{\pi(K)}$.
\end{lemma}
\begin{proof}
	Fix $y \in \intr{\pi(K)}$. Assume, to the contrary, that $f_K(y) < m$. Let $a$ and $b$, with $a<b$, be the endpoints of the segment $I_K(y)$. Since the length of $I_K(y)$ is strictly less than $m$,  the segment $I_K(y)$ can be split be split into two segments of lengths less than $1$ and less than $m-1$. That is, there exists $t$ satisfying $a < t < b$, $t -a < 1$ and $b- t < m-1$.  With this choice of $t$, one has $|I_K(y) \cap (t+\Z)|  = 1 + \floor{b-t} \le m-1$.
	
	We show that, for the point $x:=(y,t) \in \intr{K}$, one has $\mult(K,x) \le m-1$. For this, consider an arbitrary $z \in \Z^d \setminus \{o\}$ such that $x \in K + z$. We have $z \in K - x \subseteq K - \intr{K}$. By \eqref{int:sum}, $K-\intr{K}=\intr{K-K}$. Thus, $z$ is a collision vector. Hence $z$ is a multiple of $e_d$ and we can represent it as $z=(o,s)$ with $s \in \Z \setminus \{0\}$. Reformulating $x \in K+z$ as $x - z \in K$, and using $x=(y,t)$ and $z=(o,s)$, we arrive at $t-s \in I_K(y)$. Thus, $t-s \in (I_K(y) \cap (t+ \Z) ) \setminus \{t\}$. Hence, $t-s$ is one of at most $m-2$ values in  $I_K(y) \cap (t+\Z) \setminus \{t\}$. Consequently, apart from $K$, the point $x$ lies in at most $m-2$ other members of $\cT(K)$. This shows $\mult(K,x) \le m-1$. The latter contradicts the assumption that $\cT(K)$ is an $m$-fold covering. 
\end{proof}

\begin{lemma}
	\label{lem:pack:xray}
	Let $m \in \N$, let $K \subseteq \R^d$ be a convex body such that $\cT(K)$ is an $m$-fold packing. Then $f_K(y) \le m$ for every $y \in \intr{\pi(K)}$. 
\end{lemma}
\begin{proof}
	Fix $y \in \intr{\pi(K)}$. Let $a$ and $b$, with $a<b$, be the endpoints of $I_K(y)$. Assume, $f_K(y) \le m$  is not true. Then $b- a > m$. Choosing $t \in (a,b)$ sufficiently close to $a$, we ensure that the set $(a,b) \cap (t+\Z)$ contains the $m+1$ values $t, \ldots, t+m$. Setting $x:=(y,t)$, we get $x + i e_d \in \intr{K}$ for $i \in \{0,\ldots,m\}$. Thus, in view of \eqref{mult:b}, $\mult(\intr{K},x) \ge m+1$. This is a contradiction to the assumption that $\cT(K)$ is an $m$-fold packing.
\end{proof}

\begin{lemma}
	\label{lem:cylinders}
	Let $C \in \cK^d$ be $o$-symmetric and let $B:=\pi(C)$. Assume that, for some constant $\lambda>0$, one has $f_C(y) = 2 \lambda$ for every $y \in \intr{B}$. Then $C = L(B,a-\lambda,a+\lambda)$ 
	for some linear function $a : \R^{d-1} \to \R$.
\end{lemma}
\begin{proof}
	In view of \eqref{cl:int}, it suffices to verify the equality $\intr{C} = \intr{L(B,a-\lambda,a+\lambda)}$ for the interiors. Note also that 
	\begin{align*}
		\intr{C} & = \setcond{(y,t)}{y \in \intr{B}, \ t \in \intr{I_K(y)}},
		\\
		\intr{L(B,a-\lambda,a+\lambda)} & = \setcond{(y,t)}{y \in \intr{B}, \ a(y) - \lambda < t < a(y) + \lambda}.
	\end{align*}
	
	Since $f_C(o)=2\lambda$ and $C$ is $o$-symmetric, $\lambda e_d$ and $-\lambda e_d$ are boundary points of $C$. Fix a hyperplane $H$ supporting $C$ at $\lambda e_d$. The hyperplane $H$ can be described as the set of all $(y,t) \in \R^d$ satisfying $t=a(y)+\lambda$, for some linear function $a$. By the $o$-symmetry of $C$, the hyperplane $-H$ supports $C$ at $- \lambda e_d$. The hyperplane $-H$ is the set of all $(y,t) \in \R^d$ satisfying $t=a(y) - \lambda$. It follows that $I_K(y) \subseteq [a(y)- \lambda, a(y) + \lambda]$ for every $y \in \intr{B}$. Furthermore, since $[a(y) - \lambda, a(y) + \lambda]$ and $I_K(y)$ both have length $2 \lambda,$ we even have the equality $I_K(y) = [a(y) - \lambda,a(y)+\lambda]$ for every $y \in \intr{B}$. The latter implies $\intr{C}=\intr{L(B,a-\lambda,a+\lambda)}$ and by this $C=L(B,a-\lambda,a+\lambda)$.
\end{proof}

\begin{proof}[Proof of Theorem~\ref{thm:vdc:ineq:eqcase}] 
We first prove the sufficiency. Assume that $C$ is the image of $L(B,a-m,a+m))$ under a linear unimodular transformation $\phi$, where $B \in \cKo^d(1)$ is extremal and $a : \R^{d-1} \to \R$ is a linear function. Taking into account the fact that unimodular transformations preserve the volume, we get $\vol(C) = \vol(L(B,a-m,a+m))$. The volume of $L(B,a-m,a+m)$ can be computed by integration:
\begin{align}
	\label{eq:cyl:volume}
	\vol(L(B,a-m,a+m)) & = \int_B (a(y)+m)-(a(y)-m)) \dd y = 2 m \vol(B).
\end{align}
Since $B$ is a $(d-1)$-dimensional extremal body, $\vol(B) = 2^{d-1}$, and we arrive at $\vol(C)=m 2^d$. 

To prove the converse implication, assume $\vol(C)=m 2^d$.  Consider the family $\cT(\frac{1}{2} C)$. In view of   Theorem~\ref{thm:vdc:ineq}\eqref{vdc:modified}, $\dim(\Z^d \cap \intr{C})=1$. Changing coordinates by a linear unimodular transformations, we assume that all vectors in $\Z^d \cap \intr{C}$ are multiples of $e_d$. This means that $\Z^d \cap \intr{C}$ consists of the $2m-1$ vectors of the form $i e_d$ with $i \in \{-m+1,\ldots,m-1\}$. We fix $B = \pi(C)$.
	By Theorem~\ref{thm:vdc:ineq}\eqref{vdc:eq:tiling}, the equality $\vol(C) = m 2^d$ implies that $\cT(\frac{1}{2} C)$ is an $m$-fold tiling. By Lemmas~\ref{lem:cover:xray} and \ref{lem:pack:xray}, $f_{\frac{1}{2} C}(y) = m$ holds for every $y \in \intr{\frac{1}{2} B}$. Passing from $\frac{1}{2} C$ to $C$, the latter can be formulated as the equality $f_C(y) = 2 m$ for every $y \in \intr{B}$. By Lemma~\ref{lem:cylinders}, equality $C = L(B,a-m,a+m)$ holds for some linear function $a : \R^{d-1} \to \R$. It remains to show that $B$ is extremal. In view of \eqref{eq:cyl:volume}, equalities $\vol(C) = m 2^d$ and $C=L(B,a-m,a+m)$ imply $\vol(B) = 2^{d-1}$. If $B \in \cKo^{d-1}(1)$ was not true, then $\intr{B}$ would contain a point $z \in \Z^d \setminus \{o\}$. The segment $I_B(z)$ has length $2m > 1$. Consequently, $\intr{I_B(z)}$ contains an integer value $s$. We have thus constructed the non-zero lattice vector $(z,s) \in \intr{C}$, which is not a multiple of $e_d$. This is a contradiction. Thus, $B \in \cKo^{d-1}(1)$ and $\vol(B) = 2^{d-1}$, which means that $B$ is extremal.
\end{proof}

\begin{proof}[Proof of Corollary~\ref{cor:multiple:tilings}]
	A result of Gravin, Robins and Shiryaev \cite[Thm.~1.1]{MR3063154} implies that $K$ is a centrally symmetric. So, without loss of generality we can assume that $K$ is $o$-symmetric.

	\eqref{cor:mult:ineq}: The interior of $2K$ contains $2N+1$ lattice points. We thus get
	\begin{align*}
		m 2^d & = \vol(2 K) & & \text{(by Proposition~\ref{multiple:pack:cov:vol}\eqref{tiling:vol})}
		\\ & \le (N+1) 2^d & & \text{(by van der Corput's inequality)},
	\end{align*}
	which yields $m \le N+1$.

	\eqref{cor:mult:eq:case}: The characterization of the equality case $m=N+1$ is a straightforward consequence of Theorem~\ref{thm:vdc:ineq:eqcase} applied to $C=2K$.
\end{proof}

\begin{lemma}
	\label{lem:char:tiling}
	Let $m \in \N$. Let $\Lambda \subseteq \R^d$ be a lattice of rank $d$, let $K \in \cK^d$ and consider the family $\cT:=\setcond{K+z}{z \in \Lambda}$. Then $\cT$ is an $m$-fold tiling if and only if $\cT$ is an $m$-fold packing with $\vol(K)= m \det(\Lambda)$. 
\end{lemma}
\begin{proof}
	The assertion is a straightforward reformulation of the equivalence \eqref{pack:vol=m} $\Leftrightarrow$ \eqref{tiling} in Theorem~\ref{thm:tiling:char}  for the case of arbitrary rank $d$ lattices.
\end{proof}

\begin{lemma}
	\label{lem:m>=2}
	Let $m \in \N$ and let $\cT(K)$, with $K \in \cK^d$, be an $m$-fold tiling, for which the set of collision vectors is non-empty. Then $m \ge 2$. 
\end{lemma}
\begin{proof}
	We fix any collision vector $z$. Clearly, $\mult(\intr{K},x) \ge 2$ for every $x \in \intr{K} \cap (\intr{K} + z)$, where the intersection of $\intr{K}$ and $\intr{K}+z$ is non-empty. Hence $m \ge 2$. 
\end{proof}

\begin{lemma}
	\label{lem:lifting:replication}
	Let $m \in \N$ and let $C \in \cKo^d$ be a cylindrical $m$-lifting of a $(d-1)$-dimensional extremal body. Then then $m$-fold tiling $\cT(\frac{1}{2} C)$ is a replication of a one-fold tiling. 
\end{lemma}
\begin{proof}
	Without loss of generality we can assume that $C=L(B,a-m,a+m)$ for some extremal body $B \in \cKo^{d-1}(1)$ and a linear function $a: \R^{d-1} \to \R$.
	The set $(\frac{1}{2} C )\times \R$ can be decomposed into sets $\frac{1}{2} C+i m e_d$, with $i \in \Z$, that have disjoint interiors. Since $B$ is extremal, $\cT(\frac{1}{2} B)$ is a one-fold tiling. This implies that $\cT(\frac{1}{2} C)$ is a replication of the one-fold tiling $\setcond{\frac{1}{2} C + z}{z \in \Lambda}$ with $\Lambda = \Z^{d-1} \times m \Z$. 
\end{proof}

\begin{proof}[Proof of Corollary~\ref{cor:N<=2}]
	As in the proof of Corollary~\ref{cor:multiple:tilings}, without loss of generality we can assume that $K$ is $o$-symmetric. 
	
	If $N=0$, then Corollary~\ref{cor:multiple:tilings}\eqref{cor:mult:ineq} yields $m =1$. For one-fold tilings, the assertion is trivial. If $N=1$, then Corollary~\ref{cor:multiple:tilings}\eqref{cor:mult:ineq} yields $m \le 2$, while Lemma~\ref{lem:m>=2} yields $m \ge 2$. Thus, $m=2$ and $N=1$, which means that the equality $m=N+1$ holds. By Corollary~\ref{cor:multiple:tilings}\eqref{cor:mult:eq:case}, $2 K$ is a cylindrical $m$-lifting of a $(d-1)$-dimensional extremal convex body. Applying Lemma~\ref{lem:lifting:replication}, we get the desired assertion.

	If $N=2$, then Corollary~\ref{cor:multiple:tilings}\eqref{cor:mult:ineq} yields $m \le 3$, while Lemma~\ref{lem:m>=2} yields $m \ge 2$. Thus, we end up with two cases $N=2, m=2$ and $N=2, m=3$. For $N=2$, $m=3$, the equality $m=N+1$ holds, and so we can argue similarly to the case $N=1, m=2$ to verify the assertion. 
	
	It remains to consider the case $N=2$, $m=2$. We first show that the four collision vectors of $\cT(K)$ are not collinear. The set of all collision vectors can be expressed as $Z:= \intr{2 K} \cap \Z^d \setminus \{o\}$. Since $|Z| \le 4$, this shows that the convex hull $P$ of $Z$ is either a $o$-symmetric segment, with $P \cap \Z^d$ consisting of the four collinear collision vectors and the origin, or a $o$-symmetric parallelogram, with $P \cap \Z^d$ consisting of the four vertices of $P$ and the origin. If $P$ is a segment, we can assume that $P$ has endpoints $\pm 2 e_d$. Since $\cT(K)$ is a two-fold tiling,
	Lemmas~\ref{lem:cover:xray} and \ref{lem:pack:xray} imply that $f_K(y)=2$ holds for every $y \in \intr{\pi(K)}$. Lemma~\ref{lem:cylinders} shows that $K=L(B,a-1,a+1)$ for some linear function $a : \R^{d-1} \to \R$ and $B=\pi(C)$. The equality $K=L(B,a-1,a+1)$ contradicts the the fact $2 e_d$ is a collision vector of $K$, because $L(B,a-1,a+1)$ and $L(B,a-1,a+1)+ 2 e_d = L(B,a+1,a+3)$ do not have interior points in common. We have thus verified that $P$ is not a segment. 
	
	Consequently, $P$ is a parallelogram and, changing coordinates by a linear unimodular transformations, we can assume $\pm e_1, \pm e_2$ are the vertices of $P$. Consider the sub-lattice $\Lambda := \setcond{(z_1,\ldots,z_d) \in \Z^d}{z_1 + \cdots + z_d \ \text{even}}$ of $\Z^d$. It is easy to check that $\det(\Lambda)=2$. Since $\pm e_1$ and $\pm e_2$ are not in $\Lambda$, we see that $\setcond{K+z}{z \in \Lambda}$ is a packing. On the other hand,  $\cT(K)$ is a two-fold tiling, and so $\vol(K)=2$ holds, by Proposition~\ref{multiple:pack:cov:vol}\eqref{tiling:vol}. Thus, $\setcond{K+z}{z \in \Lambda}$ is a packing with $\vol(K)= \det(\Lambda)$. Lemma~\ref{lem:char:tiling} implies that $\setcond{K+ z}{z \in \Lambda}$ is a one-fold tiling.
\end{proof}




\bibliography{inputs/lit2}
\bibliographystyle{amsalpha}

\end{document}